\documentclass[9pt,twocolumn,journal]{IEEEtran}

\usepackage{amsmath, amsfonts, amsthm, amssymb, graphicx, gensymb, textcomp,enumerate}
\newcommand{\ida}{\id \left(a(s)^T x\right)^2 ds }
\renewcommand{\ll}{\left(\!\!}
\newcommand{\rr}{\!\!\right)}
\newcommand{\m}{\mathbf M}
\newcommand{\Ri}{\mathcal R_i}
\newcommand{\Rb}{\mathcal R_b}
\newcommand{\Rm}{\mathcal R_m}

\newcommand{\Rmb}{R_{m,b}}

\renewcommand{\o}{\omega}
\newcommand{\ho}{\hat \o}
\newcommand{\om}{\o_{\max}}
\renewcommand{\to}{\tilde \o}
\newcommand{\cro}[1]{[#1_{\times}]}
\renewcommand{\a}{a}

\newcommand{\ao}{\mathring{\a}}

\newcommand{\at}{\tilde \a}

\newcommand{\R}[1]{\mathbf R ^{#1}}

\newcommand{\la}{\left \{}
\newcommand{\ra}{\right \}}
\newcommand{\id}{\int_t^{t+T}}

\newcommand{\itT}{[t,t+T]}
\renewcommand{\c}{\textnormal{c}}
\newcommand{\cc}{\c\,}
\newcommand{\s}{\textnormal{s}}
\renewcommand{\ss}{\s\,}
\newtheorem{Problem}{Problem}

\newtheorem{assump}{Assumption}
\newtheorem{proposition}{Proposition}
\newtheorem{lemma}{Lemma}
\newtheorem{theorem}{Theorem}
\newtheorem{remark}{Remark}

\begin{document}


\title{Angular velocity nonlinear observer from single vector measurements}

\author{Lionel Magnis,\thanks{MINES ParisTech, PSL Research University, CAS, 60 bd Saint-Michel, 75272 Paris Cedex FRANCE.  E-mail: lionel.magnis@mines-paristech.fr .}~
Nicolas Petit\thanks{MINES ParisTech, PSL Research University, CAS, 60 bd Saint-Michel, 75272 Paris Cedex FRANCE.  E-mail: nicolas.petit@mines-paristech.fr . Phone: +33140519330}}

\maketitle

\begin{IEEEkeywords}
Sensor and data fusion; nonlinear observer and filter design; time-varying systems; guidance navigation and control.               
\end{IEEEkeywords}                             

\begin{abstract}
The paper proposes a technique to estimate the angular velocity of a rigid body from single vector measurements. Compared to the approaches presented in the literature, it does not use attitude information nor rate gyros as inputs. Instead, vector measurements are directly filtered through a nonlinear observer estimating the angular velocity. Convergence is established using a detailed analysis of a linear-time varying dynamics appearing in the estimation error equation. This equation stems from the classic Euler equations and measurement equations. As is proven, the case of free-rotation allows one to relax the persistence of excitation assumption. Simulation results are provided to illustrate the method.
\end{abstract}

\section{Introduction}
This article considers the question of estimating the angular velocity of a rigid body from signals from embedded sensors. This general question is of particular importance in various fields of engineering, and in particular for the problem of orientation control, as shown in numerous applications~\cite{salcudean1991,boskovic2000,silani2003,lovera2005} for spacecraft, unmanned aerial vehicles, guided ammunitions, to name a few.

In the literature, two types of methods have been proposed to address this question. First, one can directly measure the angular velocity by using a specific sensor. This straightforward solution requires a strap-down rate gyro~\cite{Titterton2004}. However, rate gyros being relatively fragile and expensive components, prone to drift, this solution is often discarded. The alternative is a \emph{two-step} approach. In the first step, attitude is determined from measurements of known reference vectors. Then, in the second step, attitude variations are used to estimate the angular velocity.

The first step is detailed in~\cite{crassidis2007}. In a nutshell, when two independent vectors are measured with vector sensors attached to a rigid body, the attitude of the rigid body can be found under the form  the solution of the Wahba problem~\cite{wahba1965} which is a minimization problem having as unknown the rotation matrix from a fixed frame to the body frame. Thus, at any instant, full attitude information can be obtained~\cite{shuster1978,shuster1990,baritzhack1996,choukroun2003}. In principles, this is sufficient to perform the second step: once the attitude is known, angular velocity can be estimated from a time-differentiation. However, noises disturb this process. To address this issue, introducing \emph{a priori} information in the estimation process allows one to filter-out noises from the estimates. Following this approach, numerous observers based on the Euler equations have been proposed to estimate angular velocity from full attitude information \cite{salcudean1991,thienel2007,sunde2005,jorgensen2011}.

Besides this two-step approach, which requires measurements of two independent reference vectors, a more direct and less requiring solution can be proposed. In this paper, we expose an algorithm that directly uses the measurements of a \emph{single} vector and reconstructs the angular velocity in a simple manner, by means of a nonlinear observer. This is the contribution of this article. In a related philosophy, we have recently proposed an observer using the measurements from two linearly independent vectors as input~\cite{magnis2014a}. The present paper studies a similarly structured observer. However, due to the fact that here only a single vector measurement is employed, the arguments of proof are completely different, and result in a new and independent contribution.

The paper is organized as follows. In Section~\ref{sec : notations}, we introduce the notations and the problem statement. We analyze the attitude dynamics (rotation and Euler equations) and relate it to the measurements. In Section~\ref{sec : observer}, we define the proposed nonlinear observer. The observer has an extended state and uses output injection. To prove its convergence, the error equation is identified as a linear time-varying (LTV) system perturbed by a linear-quadratic term. Under a persistent excitation (PE) assumption, the LTV dynamics is shown to generate an exponentially convergent dynamics. This property, together with assumptions on the inertia parameters of the rigid body, reveal instrumental to conclude on the exponential uniform convergence of the error dynamics. Importantly, the PE assumption is proven to be automatically satisfied in the particular case of free-rotation. In details, in Section~\ref{sec:PE}, we establish that for almost all initial conditions, the PE assumption holds. This result stems from a detailed analysis of the various types of solutions to the free-rotation dynamics. Illustrative simulation results are given in Section~\ref{sec : simu}. Conclusions and perspectives are given in Section~\ref{sec : conclusion}.
\section{Notations and problem statement}
\label{sec : notations}
\subsection{Notations}
\textbf{Vectors in $\R{3}$} are written with small letters $x$. $|x|$ is the Euclidean norm of $x$. $\cro{x}$ is the skew-symmetric cross-product matrix associated with $x$, i.e. ${\forall y \in \R{3}, \ \cro{x}y = x \times y}$. Namely,
\begin{displaymath}
\cro{x} \triangleq \left(
\begin{array}{ccc}
 0   & -x_3 &  x_2 \\
 x_3 &  0   & -x_1 \\
-x_2 &  x_1 &  0
\end{array}
\right)
\end{displaymath}
where $x_1,x_2,x_3$ are the coordinates of $x$ in the standard basis of $\R{3}$. If $x$ is a unit vector, we have
\begin{displaymath}
\cro{x}^2 = x x^T - I
\end{displaymath}

\textbf{Vectors in $\R{6}$} are written with capital letters $X$. $|X|$ is the Euclidean norm of $X$. The induced norm on $6\times6$ matrices is noted $||\cdot||$. Namely,
\begin{displaymath}
||M|| = \max_{|X|=1} |MX|
\end{displaymath}
For convenience, we may write $X$ under the form
\begin{displaymath}
X = \left(X_1^T,X_2^T\right)^T
\end{displaymath}
with $X_1,X_2 \in \R{3}$. Note that $${|X|^2 = |X_1|^2+|X_2|^2}$$

\textbf{Frames} considered in the following are orthonormal bases of $\R{3}$.

\textbf{Rotation matrix.} For any unit vector ${u \in \R{3}}$ and any ${\zeta \in \mathbf{R}}$, $r_u(\zeta)$ designates the rotation matrix of axis $u$ and angle $\zeta$. Namely
\begin{displaymath}
r_u(\zeta) \triangleq \cos \zeta I + \sin \zeta \cro{u} + (1-\cos \zeta)u u^T
\end{displaymath}

\subsection{Problem statement}
Consider a rigid body rotating with respect to an inertial frame $\Ri$. Note $R$ the rotation matrix from $\Ri$ to a body frame $\Rb$ attached to the rigid body and $\o$ the corresponding angular velocity vector, expressed in $\Rb$.
Assuming that the body rotates under the influence of an external torque $\tau$ (which, is null in the case of free-rotation), the variables $R$ and $\o$ are governed by the following differential equations
\begin{align}
\dot R & = R \cro{\o} \label{eq R}\\
\dot \o & = J^{-1}\left(J\o \times \o + \tau \right) \triangleq E(\o) + J^{-1}\tau \label{eq euler}
\end{align}
where $J = \textrm{diag}(J_1,J_2,J_3)$ is the inertia matrix\footnote{Without restriction, we consider that the axes of $\Rb$ are aligned with the principal axes of inertia of the rigid body.}. Equation~\eqref{eq euler} is known as the set of Euler equations for a rotating rigid body \cite{landau1982}. The torque $\tau$ may result from control inputs or disturbances\footnote{In the case of a satellite e.g., the torque could be generated by inertia wheels, magnetorquers, gravity gradient, among other possibilities.}. We assume that $J$ and $\tau$ are known.

We assume that a constant reference unit vector $\ao$ expressed in $\Ri$ is known, and that sensors arranged on the rigid body allow to measure the corresponding unit vector expressed in $\Rb$. Namely, the measurements are
\begin{equation}
\label{def : a,b}
\a(t) \triangleq R(t)^T \ao
\end{equation}
For implementation, the sensors could be e.g. accelerometers, magnetometers, or Sun sensors to name a few \cite{magnis2014}. We now formulate some assumptions.
\begin{assump}
\label{hypothese O borné}
$\o$ is bounded : $|\o(t)| \leq \om$ at all times
\end{assump}
\begin{assump}[persistent excitation]
\label{hyp:PE}
There exist constant parameters $T >0$ and $0 < \mu < 1$ such that $a(\cdot)$ satisfies
\begin{equation}
\label{eq:pe}
\frac{1}{T}\id \cro{a(\tau)}^T \cro{a(\tau)} d\tau \geq \mu I, \quad \forall t
\end{equation}
\end{assump}
The problem we address in this paper is the following.
\begin{Problem}
\label{Problem}
Under Assumptions~\ref{hypothese O borné}-\ref{hyp:PE}, find an estimate $\ho$ of $\o$ from the measurements $\a$ defined in \eqref{def : a,b}.
\end{Problem}
\begin{remark}[on the persistent excitation]
\label{rk:PE}
\eqref{eq:pe} is equivalent to
\begin{equation}
\label{eq:pe equivalent}
\frac{1}{T} \id \left(x^Ta(\tau)\right)^2 d\tau \leq 1-\mu, \quad \forall t, \quad \forall |x| = 1
\end{equation}
which is only possible if $a(\cdot)$ varies uniformly on every interval $[t,t+T]$. Without the PE assumption, Problem~\ref{Problem} may not have a solution. For example, the initial conditions
\begin{displaymath}
a(t_0) = \left(\begin{array}{c}
1 \\
0 \\
0
\end{array}\right), \quad
\o(t_0) = \left(\begin{array}{c}
w \\
0 \\
0
\end{array}\right)
\end{displaymath}
yield $a(t) = a(t_0)$ for all $t$, regardless of the value of $w$. Hence, the system is clearly not observable. Such a case is discarded by the PE assumption. Note that this assumption bears on the trajectory, hence on the initial condition $X(t_0)$ and on the torque $\tau$ only.
\end{remark}

\section{Observer definition and analysis of convergence}
\label{sec : observer}
\subsection{Observer definition}
The time derivative of the measurement $\a$ is
\begin{displaymath}
\label{dot ab}
\dot \a = \dot R^T \ao = - \cro{\o} R^T \ao = \a \times \o
\end{displaymath}
To solve Problem~\ref{Problem}, the main idea of the paper is to consider the reconstruction of the extended 6-dimensional state $X$ by its estimate $\hat X$
\begin{displaymath}
X = \left(
\begin{array}{c}
\a \\
\o
\end{array}\right), \quad
\hat X = \left(
\begin{array}{c}
\hat a \\
\ho
\end{array}\right)
\end{displaymath}
The state is governed by
\begin{displaymath}
\dot X = \left(
\begin{array}{c}
\a \times \o\\
E(\o) + J^{-1}\tau
\end{array}
\right)
\end{displaymath}
and the following observer is proposed
\begin{equation}
\label{def : observer}
  \dot {\hat X}  =   \left(
\begin{array}{c}
\a \times \ho - k(\hat \a - \a)\\
E(\ho) + J^{-1}\tau +  k^2  \a  \times  (\hat \a - \a)
\end{array}
\right)
\end{equation}
where $k>0$ is a constant (tuning) parameter. Note
\begin{equation}
\label{def erreur}
\tilde X \triangleq X - \hat X \triangleq \left(
\begin{array}{c}
\at \\
\to
\end{array}\right)
\end{equation}
the error state. We have
\begin{equation}
\label{eq : dot X tilde}
\dot{\tilde X} = \left(
\begin{array}{cc}
-k I & \cro{\a} \\
k^2 \cro{\a} & 0
\end{array}
\right)
  \tilde X
+
\left(
\begin{array}{c}
0 \\
E(\o) - E(\ho)
\end{array}
\right)
\end{equation}

\subsection{Preliminary change of variables and properties}
The study of the dynamics \eqref{eq : dot X tilde} employs a preliminary change of coordinates. Note
\begin{equation}
\label{def : Z}
Z \triangleq \left(
\begin{array}{c}
\tilde \a \\
\frac{\to}{k}
\end{array}\right)
\end{equation}
yielding
\begin{displaymath}
\label{systeme Z perturbe}
\dot{Z} = kA(t)Z
+
\ll
\begin{array}{c}
0 \\
\frac{E(\o) - E(\ho)}{k}
\end{array}
\rr
\end{displaymath}
with
\begin{equation}
\label{def : A}
A(t) \triangleq \left(
\begin{array}{ccc}
- I & \cro{\a(t)} \\
\cro{\a(t)} & 0
\end{array}
\right)
\end{equation}
which we will analyze as an ideal linear time-varying (LTV) system
\begin{equation}
\label{systeme Z non perturbe}
\dot Z = k A(t) Z
\end{equation}
perturbed by the input term
\begin{equation}
\label{def : xi}
\xi \triangleq \ll
\begin{array}{c}
0 \\
\frac{E(\o) - E(\ho)}{k}
\end{array}
\rr
\end{equation}
We start by upper-bounding the disturbance \eqref{def : xi}.
\begin{proposition}[Bound on the disturbance]
For any $Z$, $\xi$ is bounded by
\begin{equation}
\label{prop : borne xi}
|\xi| \leq d(\sqrt{2}\om |Z| + k |Z|^2)
\end{equation}
where $d$ is defined as
\begin{equation}
\label{def : d}
d \triangleq \max \la \left|\frac{J_3-J_2}{J_1}\right|, \quad \left|\frac{J_1-J_3}{J_2}\right|, \quad \left|\frac{J_2-J_1}{J_3}\right| \ra
\end{equation}
\end{proposition}
\begin{proof}
We have $$|\xi| = \frac{1}{k} |E(\o) - E(\ho)|$$
with, due to the quadratic nature of $E(\cdot)$,
\begin{align*}
& E(\o) - E(\ho) =  J^{-1} \left( J\to \times \o + J\o \times \to - J\to \times \to \right) \\
& =
\left(
\begin{array}{c}
\frac{J_2-J_3}{J_1} (\o_2 \to_3 + \to_2 \o_3) \\
\frac{J_3-J_1}{J_2} (\o_3 \to_1 + \to_3 \o_1) \\
\frac{J_1-J_2}{J_3} (\o_1 \to_2 + \to_1 \o_2) \\
\end{array}
\right)
-
\left(
\begin{array}{c}
\frac{J_2-J_3}{J_1} \to_2 \to_3 \\
\frac{J_3-J_1}{J_2} \to_3 \to_1 \\
\frac{J_1-J_2}{J_3} \to_1 \to_2 \\
\end{array}
\right) \\
& \triangleq \delta_1 - \delta_2
\end{align*}
As a straightforward consequence
\begin{displaymath}
|\delta_2| \leq d|\to|^2
\end{displaymath}
Moreover, by Cauchy-Schwarz inequality
\begin{displaymath}
(\o_2 \to_3 + \to_2 \o_3)^2 \leq (\o_2^2+\o_3^2) (\to_2^3 + \to_3^2) \leq (\o_2^2+\o_3^2) |\to|^2
\end{displaymath}
Using similar inequalities for all the coordinates of $\delta_1$ yields
\begin{displaymath}
|\delta_1|^2 \leq 2 d^2 |\o|^2 |\to|^2 \leq 2 d^2\om^2 |\to|^2
\end{displaymath}
Hence,
\begin{align*}
|\xi| \leq \frac{|\delta_1| + |\delta_2|}{k}&  \leq d \sqrt{2} \om \left|\frac{\to}{k}\right| + k d \left|\frac{\to}{k}\right|^2 \\
	  & \leq d (\sqrt{2} \om |Z| + k  |Z|^2)
\end{align*}
\end{proof}
\begin{remark}[on the quantity $d$]
As $J_1,J_2,J_3$ are the main moments of inertia of the rigid body, we have \cite{landau1982} (\S 32,9) $$J_i \leq J_j + J_k$$ for all permutations $i,j,k$ and hence ${0 \leq d \leq 1}$. Moreover, $d=0$ if and only if ${J_1=J_2=J_3}$. $d$ appears as a measurement of how far the rigid body is from an ideal symmetric body. For this reason, we call it \emph{distordance} of the rigid body. Examples:

\begin{itemize}
\item For a homogeneous parallelepiped of size ${l \times l \times L}$, with ${L \geq l}$, we have $$d = \frac{L^2-l^2}{L^2 + l^2}$$
\item For a homogeneous straight cylinder of radius $r$ and height $h$ we have $$d = \frac{\left|h^2-3r^2\right|}{h^2+3r^2}$$
\end{itemize}
\end{remark}

\subsection{Analysis of the LTV dynamics $\dot Z = k A(t) Z$}
\label{sec : Z'=kAZ}
The shape of $A(t)$ will appear familiar to the reader acquainted with adaptive control problems. Along the trajectories of~\eqref{systeme Z non perturbe} we have
\begin{displaymath}
\frac{d}{dt} |Z|^2 = -2k|Z_1|^2 = -Z^T C^T C Z
\end{displaymath}
with $${C \triangleq (\sqrt {2k} \quad 0)}$$ As will be seen in the proof of the following Theorem, the PE assumption will imply, in turn, that the pair $(kA(\cdot),C)$ is uniformly completely observable (UCO), which guarantees uniform exponential stability of the LTV system.

\begin{theorem}[LTV system exponential stability]
\label{thm Z' = kAZ}
There exists $0<c<1$ depending only on $T, \mu,k$ and $\om$ such that the solution of \eqref{systeme Z non perturbe} satisfies for all integer $N\geq 0$ $$|Z(t)|^2 \leq c^{N} |Z(t_0)|^2, \quad \forall t\in [t_0+NT,t_0+(N+1)T]$$
for any initial condition $t_0, Z(t_0)$.
\end{theorem}
\begin{proof}
Along the trajectories of~\eqref{systeme Z non perturbe} we have
\begin{displaymath}
\frac{d}{dt} |Z|^2 = -2k|Z_1|^2 \leq 0
\end{displaymath}
which proves the result for $N = 0$. For all $t$
\begin{displaymath}
|Z(t+T)|^2 = |Z(t)|^2 - Z(t)^2 W(t,t+T) Z(t)
\end{displaymath}
where
\begin{displaymath}
W(t,t+T) \triangleq \id \phi(\tau,t)^T C^T C \phi(\tau,t) d\tau
\end{displaymath}
is the observability Gramian of the pair $(kA(\cdot),C)$ and $\phi$ is the transition matrix associated with~\eqref{systeme Z non perturbe}. Computing $W$ is no easy task. However, the output injection UCO equivalence result presented in \cite{ioannou1995} allows us to consider a much simpler system. Note
\begin{displaymath}
K(t) \triangleq \frac{\sqrt{k}}{\sqrt{2}} \left(\begin{array}{c}
I \\
-\cro{a(t)}
\end{array}\right)
\end{displaymath}
and
\begin{align*}
M(t) & \triangleq kA(t)+K(t) C \\
& = \left(\begin{array}{cc}
0 & k\cro{a(t)} \\
0 & 0
\end{array}\right)
\end{align*}
The observability Gramian $\widetilde W$ of the pair ${(M(\cdot),C)}$ is easily computed as
\begin{displaymath}
\widetilde W(t,t+T) = 2k \id \left(\begin{array}{cc}
I & \mathcal A(\tau,t) \\
\mathcal A(\tau,t)^T & \mathcal A(\tau,t)^T \mathcal A(\tau,t)
\end{array}\right)
d\tau
\end{displaymath}
where
\begin{displaymath}
\mathcal A(\tau,t) \triangleq k \int_t^{\tau} \cro{a(u)} du
\end{displaymath}
Such a Gramian is well known in optimal control and has been extensively studied e.g. in \cite{khalil1996}, Lemma~13.4. We have
\begin{itemize}
\item $\id k\cro{a(\tau)}^T k\cro{a(\tau)} d\tau \geq T k^2 \mu I, \quad \forall t$
\item $k\cro{a(\cdot)}$ is bounded by $k$
\item $\frac{d}{dt}k\cro{a(\cdot)}$ is bounded by $k\om$
\end{itemize}
from which we deduce that there exists ${0 < \beta_1 < 1}$ depending on $T,\mu,k,\om$ such that
\begin{displaymath}
\widetilde W(t,t+T) \geq \beta_1 I, \quad \forall t
\end{displaymath}
There also exists $\beta_2 > 0$ depending on $k,T$ such that ${\widetilde W(t,t+T) \leq \beta_2 I}$. From \cite{ioannou1995}, Lemma~4.8.1 (output injection UCO equivalence), $W(t,t+T)$ is also lower-bounded. More precisely, we have
\begin{displaymath}
W(t,t+T) \geq \frac{\beta_1}{2(1+\beta_2 T k)}I \triangleq (1-c) I
\end{displaymath}
with $0 < c < 1$. Assume the result is true for an integer ${N \geq 0}$. For any ${t \in [t_0+NT, t_0+(N+1)T]}$ we have
\begin{align*}
|Z(t+T)|^2 & = |Z(t)|^2 - Z(t)^TW(t,t+T) Z(t) \\
& \leq c |Z(t)|^2 \leq c^{N+1} |Z(t_0)|^2
\end{align*}
which concludes the proof by induction.
\end{proof}
\subsection{Convergence of the observer}
\label{sec : convergence}
Consider the quantity
\begin{equation}
\label{def : d*}
d^* \triangleq \frac{1-c}{2\sqrt 2 T \om}
\end{equation}
where $c$ is defined in Theorem~\ref{thm Z' = kAZ}. The following Theorem, which is the main result of the paper, shows that if ${d < d^*}$, the observer~\eqref{def : observer} gives a solution to Problem~\ref{Problem}.
\begin{theorem}[main result]
\label{thm convergence obs}
We suppose that Assumptions~\ref{hypothese O borné}-\ref{hyp:PE} are satisfied and that
\begin{displaymath}
d < d^*
\end{displaymath}
where $d^*$ is defined in~\eqref{def : d*}. The observer~\eqref{def : observer} defines an error dynamics~\eqref{eq : dot X tilde} for which the equilibrium 0 is locally uniformly exponentially stable. The basin of attraction of this equilibrium contains the ellipsoid
\begin{equation}
\label{bassin d'attraction}
\la \tilde X(0), \quad |\tilde \a (0)|^2 + \frac{|\to(0)|^2}{k^2} < r^2 \ra
\end{equation}
with
\begin{equation}
\label{eq : def r}
r^2 \triangleq  \frac{(1-c)^3}{8\sqrt{3}d^2T^3k^3}\left(1-\frac{2\sqrt 2 dT \om}{1-c}\right)^2
\end{equation}
\end{theorem}

\begin{proof}
Consider the candidate Lyapunov function
\begin{displaymath}
V(t,Z) \triangleq Z^T \left( \int_t^{+\infty} \phi(\tau,t)^T \phi(\tau,t) d\tau \right)Z
\end{displaymath}
where $\phi$ is the transition matrix of system~\eqref{systeme Z non perturbe}. Let $(t,Z)$ be fixed. One easily shows that $kA(\cdot)$ is bounded by $k\sqrt 3$. Thus (see for example \cite{khalil2000} Theorem 4.12)
\begin{displaymath}
V(t,Z) \geq \frac{1}{2k\sqrt 3} |Z|^2 \triangleq c_1 |Z|^2 \triangleq W_1(Z)
\end{displaymath}
Moreover, Theorem~\ref{thm Z' = kAZ} implies that
\begin{align*}
V(t,Z) & = \sum_{N=0}^{+\infty} \int_{t+NT}^{t+(N+1)T} Z^T \phi(\tau,t)^T \phi(\tau,t) Z \\
	   & \leq T \sum_{N = 0}^{+\infty} c^N |Z|^2 = \frac{T}{1-c} |Z|^2 \\
       & \triangleq c_2 |Z|^2 \triangleq W_2(Z)
\end{align*}
By construction, $V$ satisfies
\begin{displaymath}
\frac{\partial V}{\partial t}(t,Z) + \frac{\partial }{\partial Z}V(t,Z) kA(t) Z = -|Z|^2
\end{displaymath}
Hence, the derivative of $V$ along the trajectories of \eqref{systeme Z perturbe} is
\begin{align*}
\frac{d}{dt} V(t,Z) & = -|Z|^2 + \frac{\partial V}{\partial Z}(t,Z) \ \xi
\end{align*}
Using
\begin{align*}
\left|\frac{\partial}{\partial Z} V(t,Z)\right| & = 2 \left| \int_t^{+\infty} \phi(\tau,t)^T \phi(\tau,t) d\tau Z \right| \leq \frac{2T}{1-c} |Z|
\end{align*}
together with inequality \eqref{prop : borne xi} yields
\begin{displaymath}
\left|\frac{\partial V}{\partial Z}(t,Z) \ \xi \right| \leq \frac{2dT}{1-c} \left(\sqrt 2 \o_{\max} |Z|^2 + k |Z|^3\right)
\end{displaymath}
Hence
\begin{align*}
\frac{d}{dt} V(t,Z) & \leq - |Z|^2\left(1 - \frac{2\sqrt 2 dT \om}{1-c} - \frac{2dTk}{1-c}|Z|\right) \\
& \triangleq -W_3(Z)
\end{align*}
By assumption $d < d^*$, which implies
\begin{align*}
1-\frac{2\sqrt 2dT \om}{1-c} > 0
\end{align*}
We proceed as in \cite{khalil2000} Theorem~4.9. If the initial condition of \eqref{systeme Z perturbe} satisfies
\begin{align*}
|Z(t_0)| & < r  \\
\Leftrightarrow |Z(t_0)| & < \frac{1-c}{2dkT} \left(1-\frac{2 \sqrt 2 dT\om}{1-c}\right) \times \sqrt{\frac{c_1}{c_2}}
\end{align*}
then $W_3(Z(t_0)) >0$ and,  while $W_3(Z(t)) >0$, $Z(\cdot)$ remains bounded by
\begin{align*}
|Z(t)|^2 & \leq \frac{V(t)}{c_1} \leq \frac{V(t_0)}{c_1} \leq \frac{c_2}{c_1} |Z(t_0)|^2
\end{align*}
which shows that
\begin{align*}
W_3(Z) \geq \left(1 - \frac{2 \sqrt 2 dT \om}{1-c} - \frac{2dkT}{1-c} \sqrt{\frac{c_2}{c_1}}|Z(t_0)|\right) |Z|^2
\end{align*}
From \cite{khalil2000}, Theorem~4.10, \eqref{systeme Z perturbe} is locally uniformly exponentially stable. From \eqref{def : Z}, one directly deduces that the basin of attraction contains the ellipsoid \eqref{bassin d'attraction}.
\end{proof}
\begin{remark}
The limitations imposed on $\tilde \a(0)$ in \eqref{bassin d'attraction} are not truly restrictive because, as the actual value $\a(0)$ is assumed known, the observer may be initialized with ${\at(0) = 0}$. What matters is that the error on the unknown quantity $\o(0)$ can be large in practice.
\end{remark}

\section{PE assumption in free-rotation}
\label{sec:PE}
The PE Assumption~\ref{hyp:PE} is the cornerstone of the proof of the main result. It is interesting to investigate whether it is often satisfied in practice (we have seen in Remark~\ref{rk:PE} that it might fail). In this section we consider a free-rotation dynamics, namely $\tau = 0$. We will prove that Assumption~\ref{hyp:PE}, or equivalently condition~\eqref{eq:pe equivalent}, is satisfied for almost all initial conditions.

The following important properties hold.
\begin{itemize}
\item ${\omega^T J \omega}$ is constant over time (which implies that Assumption~\ref{hypothese O borné} is satisfied)
\item The \emph{moment of inertia} of the rigid body expressed in the inertial frame
\begin{equation}
\label{def : M}
\m \triangleq R(t)J\omega(t)
\end{equation}
is constant over time.
\item Thus, any trajectory ${t \mapsto \omega(t)}$ lies on the intersection of two ellipsoids
\begin{displaymath}
\omega^T J \omega = \omega(t_0)^T J \omega(t_0), \quad \omega^T J^2 \omega = \omega(t_0)^TJ^2 \omega(t_0)
\end{displaymath}
\end{itemize}
The analysis of the intersection of those ellipsoids is quite involved and has been extensively studied in e.g. \cite{landau1982}. It follows that there are four kinds of trajectories for the solutions $\omega$ of \eqref{eq euler}.  We list them below, where $(\omega_1,\omega_2,\omega_3)$ are the coordinates of $\omega$ in the body frame.
\begin{description}
\item [Type 1]~ $\omega$ is constant, which is observed if and only if $\omega(t_0)$ is an eigenvector of $J$.
\item [Type 2]~ ${J_1>J_2>J_3}$ \emph{singular case}: $\omega_1(t)$ and $\omega_3(t)$ vanish, $\omega_2(t)$ tends to a constant when $t$ goes to infinity. This situation is observed only for a zero-measure set of initial condition $\omega(t_0)$.
\item [Type 3]~ ${J_1>J_2>J_3}$ \emph{regular case}: the trajectory is periodic and not contained in a plane. This situation is observed for almost all initial condition $\omega(t_0)$.
\item [Type 4]~ the trajectory is periodic and draws a non-zero diameter circle. This situation is observed if and only if two moments of inertia are equal and $\omega(t_0)$ is not an eigenvector of $J$.
\end{description}
Examples of such trajectories are given in Figures~\ref{fig:types 123}-\ref{fig:type 4} for various initial conditions.

\begin{figure}[!ht]
\includegraphics[width=\columnwidth]{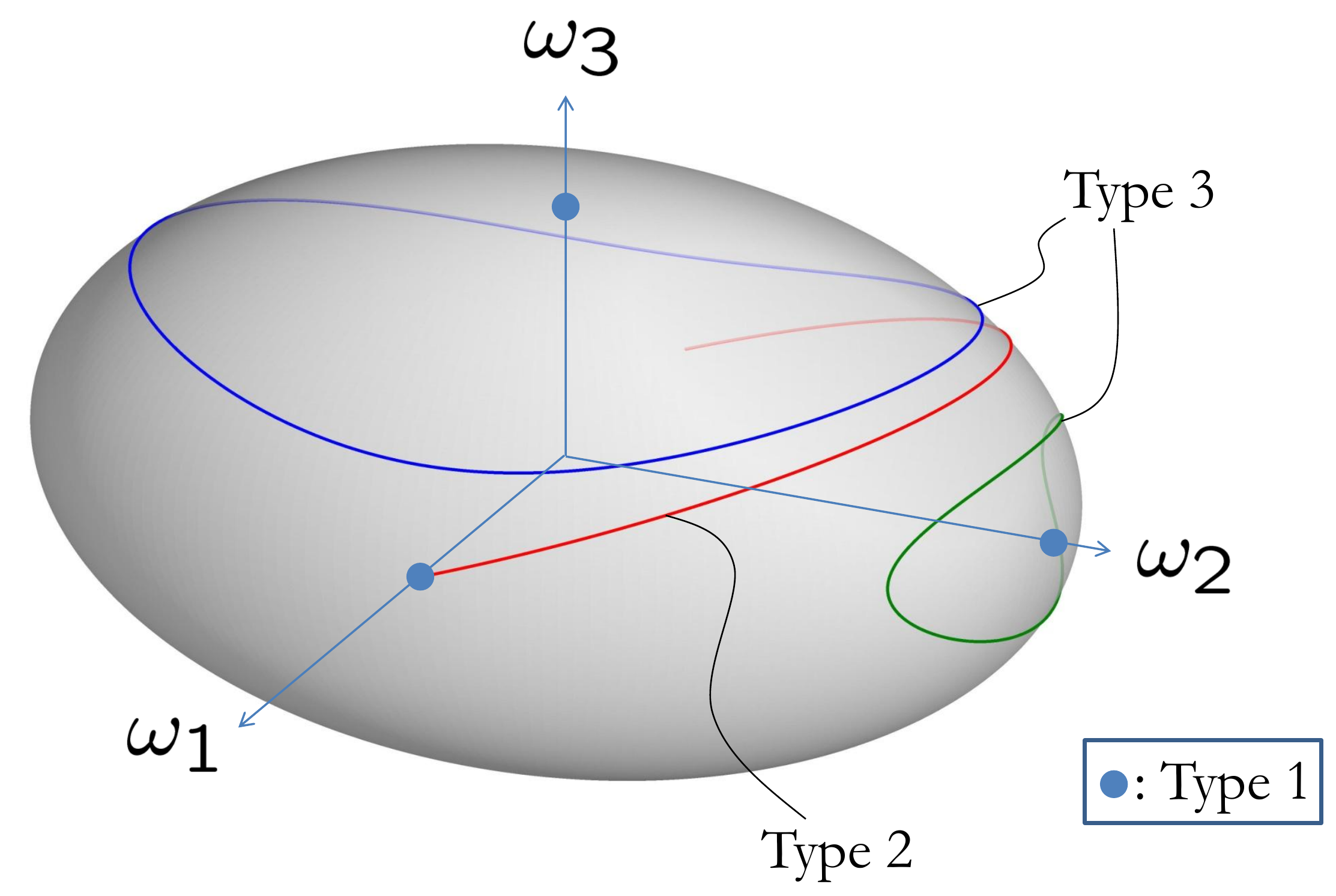}
\caption{Types 1, 2 and 3 trajectories in the case $J_1 > J_2 > J_3$ on an ellipsoid corresponding to a constant $|\m|$}
\label{fig:types 123}
\end{figure}

\subsection{Study of Type 1 and Type 2 solutions}
\label{sec : planar}
The simplest case one can imagine is when $\omega(t_0)$ (or simply $\omega$) is an eigenvector of $J$, namely for $i=1,2$ or $3$
\begin{displaymath}
J\omega = J_i\omega
\end{displaymath}
Note
\begin{displaymath}
R_0 \triangleq R(t_0), \quad w \triangleq |\omega|, \quad u \triangleq \frac{1}{w} R_0\omega
\end{displaymath}
\begin{proposition}
For all $t$, $R(t)$ writes
\begin{align*}
R(t) & = r_{u}(wt) R_0 \nonumber \\
& = \left(\cc wt I + \ss wt  \cro{u} + (1-\cc wt ) u u^T\right)R_0
\end{align*}
where $\c, \, \s$ stand for $\cos, \, \sin$ respectively.
\end{proposition}
\begin{proof}
$R(t)$ and $r_u(wt)R_0$ have the same value $R(t_0)$ for $t = t_0$. Moreover,
\begin{align*}
\frac{d}{dt} r_{u}(wt)R_0 & = w \left(-\ss wt I + \cc wt \cro{u} + \ss wt u u^T\right)R_0 \\
& = \left(\cc wt \cro{u} + \ss wt \cro{u}^2 \right) w R_0 \\
& = \left(\cc wt I + \ss wt \cro{u}\right) w \cro{u} R_0 \\
& =  \left(\cc wt I + \ss wt \cro{u} + (1-\cc wt)u u^T\right)w \cro{u} R_0 \\
& = r_{u}(wt) w \cro{u} R_0 \\
& = r_u(wt)\cro{R_0\omega} R_0 = r_u(wt) R_0 \cro{\omega}
\end{align*}
Thus both functions satisfy~\eqref{eq R}, which concludes the proof by Cauchy-Lipschitz uniqueness theorem.
\end{proof}
It follows that for all $t$, $a(t)$ writes,
\begin{align}
\label{eq : a(t) planar}
a(t) & = R(t)^T \ao \\
& =  \cc wt R_0^T\ao - \ss wt R_0^T (u \times \ao) + (1-\cc wt) u^T\ao R_0^Tu \nonumber
\end{align}
For this reason, we call \emph{planar rotation} the $R(\cdot)$ matrix generated by a Type 1 trajectory.
\begin{remark}
\label{rk : m et u}
The direction $u$ of the rotation can be simply computed from $\m$. We have
\begin{displaymath}
\m = RJ\omega =  J_i R \omega = w J_i R R_0^T u = w J_iu
\end{displaymath}
which implies that
\begin{displaymath}
u = \frac{\m}{|\m|}
\end{displaymath}
\end{remark}
The impact of the planar nature of the rotation on the PE assumption is as explained in the next two subsections.
\subsubsection{Type 1 solution with $\m$ aligned with $\ao$}
\label{sec:planar non pe}
Consider that $\ao$ is aligned with $\m = R(t_0)J\omega(t_0)$. In this case $u = \pm \ao$ (see Remark~\ref{rk : m et u}). Thus, \eqref{eq : a(t) planar} yields ${a(t) = R_0^T \ao}$ constant over time. For any $T$ we have, for the unit vector ${x = R_0^T\ao}$
\begin{displaymath}
\frac{1}{T} \int_0^T (a(s)^Tx)^2 ds = 1
\end{displaymath}
Thus, condition~\eqref{eq:pe equivalent} is not satisfied.

\subsubsection{Type 1 solution with $\m$ not aligned with $\ao$}
\label{sec:planar pe}
Conversely, consider that $\ao$ is not aligned with~$\m$. Define $v,z$ such that $(u,v,z)$ is a direct orthonormal basis of $\R{3}$. The decomposition of the unit vector $\ao$ in this basis is given as
\begin{displaymath}
\ao = a_1 u + a_2 v + a_3 z, \quad a_1^2 + a_2^2 + a_3^2=1, \quad \textrm{with} \quad a_1^2 < 1
\end{displaymath}
We have
\begin{align*}
a(t)= & R_0^T(a_1 u + (a_2 \cc wt + a_3 \ss wt)v + (a_3 \cc wt - a_2 \ss wt)z)
\end{align*}
For $T=\frac{2\pi}{w}$, any $t$ and any unit vector $${x = R_0^T(x_1 u + x_2 v + x_3 z)}$$ we have
\begin{align*}
& \frac{1}{T} \ida = \\
& \frac{1}{T} \!\! \id \!\!\!\!\!\!\!\!\!\!\! \left(a_1 x_1 + (a_2 \cc wt + a_3 \ss wt)x_2 + (a_3 \cc wt - a_2 \ss wt)x_3\right)^2 \!ds \\
& = a_1^2 x_1^2 + \frac{a_2^2+a_3^2}{2}(x_2^2+x_3^2) \leq (1-\mu)
\end{align*}
with
\begin{displaymath}
\mu \triangleq \min \left(1-a_1^2,\frac{1+a_1^2}{2}\right) \in (0,1)
\end{displaymath}
Thus, condition~\eqref{eq:pe equivalent} is satisfied.

\subsubsection{Type 2 solutions}
As shown in \cite{landau1982}, the Type~2 solutions are characterized by $J_1>J_2>J_3$ and
\begin{displaymath}
\left|\omega_1(t_0)\right| = \sqrt{\frac{J_3(J_2-J_3)}{J_1(J_1-J_2)}} \left|\omega_3(t_0)\right| \neq 0
\end{displaymath}
which defines a zero-measure set. For this reason, they are called \emph{singular} solutions. In this case, $\omega(t)$ converges to a limit $\omega_{\infty} = (0,\pm w,0)$ when $t$ goes to infinity. The rotation $R(t)$ is thus asymptotically arbitrarily close to a planar rotation around $\m = R(t_0)J\omega(t_0)$. The same arguments as in Sections~\ref{sec:planar non pe},~\ref{sec:planar pe} show that condition~\eqref{eq:pe equivalent} is satisfied unless $R(t_0)J\omega(0)$ and $\ao$ are aligned.

\subsection{Study of Type 3 and Type 4 solutions}
\label{sec:types 3 et 4}
In this section we will show that the Type 3 and Type 4 solutions satisfy the PE assumption. Both proofs relies on the following technical result.
\begin{proposition}[preliminary result]
\label{prop:preliminaire}
If condition~\eqref{eq:pe} is not satisfied, then for all $T>0$ and all $\varepsilon>0$ small enough, there exists $t$ such that for all ${y\in\R{3}}$, and all ${s \in [t,t+T]}$,
\begin{itemize}
\item $R(s)y$ remains between two planes orthogonal to $\ao$ and distant by $\varepsilon|y|$
\item $R(s)^Ty$ remains between two parallel planes distant by $\varepsilon|y|$.
\end{itemize}
\end{proposition}
\begin{proof}
Consider $T > 0$ and  $\mu$ such that
\begin{equation}
\label{eq : mu max}
0< \mu < \min\left(\frac{1}{4T\om}, \frac{T\om}{4} \right) < 1
\end{equation}
Assume that \eqref{eq:pe equivalent} is not satisfied. There exists $t,x$ such that $|x| = 1$ and \begin{equation}
\label{eq : non PE mu}
\frac{1}{T} \ida \geq 1-\mu
\end{equation}
As will appear, one can use the bounded variations of $a(\cdot)$ due to its governing dynamics to establish a lower bound on the integrand. Note
\begin{displaymath}
h(s) \triangleq \left(a(s)^Tx\right)^2, \quad \forall s
\end{displaymath}
We will now show by contradiction that
\begin{displaymath}
h(s) \geq 1- 2 \sqrt{T\om \mu}, \quad \forall s \in \itT
\end{displaymath}
Assume that there exists $s_0$ such that $${h(s_0) < 1-2 \sqrt{T\om \mu}}$$ We have, for all $s$,
\begin{align*}
|\dot h(s)| & = \left|2 \dot a(s)^T x a(s)^T x\right| \\
& = \left|2(a(s)\times \omega)^T x a(s)^T x\right| \leq 2 \om
\end{align*}
Assume $s_0 \leq t+\frac{T}{2}$ and note
\begin{displaymath}
s_1 \triangleq s_0+\sqrt{\frac{T \mu}{\om}} \leq t+T
\end{displaymath}
We have, for any ${s \in [s_0,s_1]\subset [t,t+T]}$
\begin{align*}
h(s) & \leq h(s_0) + 2\om (s-s_0) \\
& < 1-2\sqrt{T\om\mu} + 2\om (s-s_0)
\end{align*}
Hence
\begin{align*}
 \ida & < 1-\frac{1}{T}\sqrt{\frac{T\mu}{\om}} \\
+ \frac{1}{T}\! \int_{s_0}^{s_1} \!\!\! & \left(1-2\sqrt{T\om \mu} + 2\om (s-s_0) \right)\! ds \\
& = 1-2\mu + \mu = 1-\mu
\end{align*}
which contradicts \eqref{eq : non PE mu}. The case ${s_0 > t+\frac{T}{2}}$ is analog with ${s \in [s_0-\sqrt{\frac{T\mu}{\om}},s_0]\subset [t,t+T]}$. Finally, we have, for all $s$
\begin{displaymath}
0 < 1-2\sqrt{T\om \mu} \leq \left(a(s)^Tx\right)^2 \leq 1
\end{displaymath}
which shows that the continuous function ${s\mapsto a(s)^Tx}$ is of constant sign, strictly positive without loss of generality. Thus, we have
\begin{displaymath}
0 < 1-2\sqrt{T\om \mu} \leq a(s)^Tx \leq 1
\end{displaymath}
and in turn
\begin{equation}
\label{def : c}
|a(s)-x|^2 = 2-2a(s)^Tx \leq 4\sqrt{T\om \mu} \triangleq \gamma \sqrt{\mu}
\end{equation}
Note $R_1$ a rotation matrix so that $${\ao = R_1 x}$$ and, for all $s$, ${u(s), \xi(s)}$ such that
\begin{align*}
R(s) & \triangleq r_{u(s)}(\xi(s))R_1 \\
\end{align*}
Note that $R(s)x = r_{u(s)}(\xi(s)) \ao$. The next Lemma formulates that the rotation $R(s)$ is uniformly close to $r_{\ao}(\xi(s))R_1$.
\begin{lemma}
\label{lemme : R presque planaire}
We have, for all $s\in[t,t+T]$ and all $y\in\R{3}$
\begin{equation}
\label{eq:R presque planaire}
\left|R(s) y - r_{\ao}(\xi(s))R_1 y\right|^2 \leq 30 \gamma \sqrt{\mu} |y|^2
\end{equation}
where $\gamma$ is defined by~\eqref{def : c}.
\end{lemma}
\begin{proof}
Let ${s\in[t,t+T]}$. For clarity we may omit the $s$ dependency of $u$ and $\xi$. Note
\begin{align*}
\Delta & \triangleq R(s) - r_{\ao}(\xi) R_1 \\
& = \left(\sin \xi \left(\cro{u} - \cro{\ao}\right) + (1-\cos \xi) \left(u u^T - \ao \ao^T\right)\right) R_1
\end{align*}
If $\ao = u(s)$, ${\|\Delta\| = 0 \leq 30\gamma\sqrt{\mu}}$. Otherwise, for $A=x$ we have, from~\eqref{def : c}
\begin{align*}
\left|\Delta A\right |^2 & = \left|R(s)x-r_{\ao}(\xi)R_1x\right|^2 = \left|R(s)x-\ao\right|^2 \\
& = \left|x-R^T(s)\ao\right|^2 = |x-a(s)|^2 \leq \gamma\sqrt{\mu}
\end{align*}
Note $v,z$ so that ${(u,v,z)}$ is an orthonormal basis of $\R{3}$ write
\begin{displaymath}
\ao = a_1 u + a_2 v + a_3 w, \quad a_1^2+a_2^2+a_3^2 = 1
\end{displaymath}
We have
\begin{align*}
\gamma\sqrt \mu & \geq \left|R(s)x - \ao\right|^2 = \left|(r_u(\xi)-I) \ao\right|\\
& = \left|(a_2(\cc \xi-1) - a_3 \ss \xi) v + (a_2\ss \xi + a_3(\cc \xi-1) )w\right|^2 \\
& = 4 (a_2^2+a_3^2) \sin^2 \frac{\xi}{2}
\end{align*}
Now, for $B = \frac{u \times x}{|u \times x|}$ we have
\begin{align*}
\left|\Delta B\right|^2 & = \frac{\sin ^2 \xi}{x_2^2+x_3^2} \left|u \times (u \times x)-x \times (u \times x)\right|^2 \\
& = \frac{\sin^2 \xi}{x_2^2+x_3^2}(1-x^Tu)^2 |u+x|^2 \\
& = \frac{\sin^2 \xi}{x_2^2+x_3^2}4(1-x_1^2)^2 \\
& \leq 16(x_2^2+x_3^2) \sin^2 \frac{\xi}{2} \leq 4\gamma\sqrt{\mu}
\end{align*}
For ${C = A \times B}$ we have
\begin{align*}
\left|\Delta C\right|^2 & = \left|QR(s)(A \times B) - P(s)(A \times B)\right|^2 \\
& = \left|QR(s)A \times QR(s) B - P(s) A \times P(s) B\right|^2 \\
& = \left|QR(s)A \times \Delta B + \Delta A \times P(s) B\right|^2 \\
& \leq 2(c\sqrt{\mu} + 4c \sqrt{\mu}) = 10 \gamma \sqrt{\mu}
\end{align*}
Finally, for any unit vector ${y = y_1A + y_2 B + y_3 C}$ we have
\begin{align*}
\left|\Delta y\right|^2 & = \left|y_1 \Delta A + y_2 \Delta B + y_3 \Delta C\right|^2 \\
& \leq 3 \left(y_1^2 \left|\Delta A\right|^2 + y_2^2 \left |\Delta B \right|^2 + y_3^2 \left|\Delta C\right|^2\right) \\
& \leq 3 (y_1^2+y_2^2+y_3^2) 10 \gamma\sqrt{\mu} = 30\gamma\sqrt{\mu} |y|^2
\end{align*}
which concludes the proof of Lemma~\ref{lemme : R presque planaire}.
\end{proof}
Note $\varepsilon = 2\sqrt{30\gamma\sqrt \mu}$ and consider any $y$ in $\R{3}$ and any $s$ in ${\itT}$. On the one hand, $r_{\ao}(\xi(s))R_1y$ lies on a circle orthogonal to $\ao$. On the other hand, $$\left|R(s)y - r_{\ao}(\xi(s))R_1y\right| \leq \frac{\varepsilon}{2}|y|$$
This yields the first item of Proposition~\ref{prop:preliminaire} as ${\mu > 0}$ is arbitrary small. Rewriting the result of Lemma~\ref{lemme : R presque planaire} as
\begin{displaymath}
\left|R_1^Tr_{\ao}(-\xi(s))y - R(s)^Ty\right|^2 \leq 30 \gamma\sqrt{\mu}|y|^2
\end{displaymath}
for any $s \in \itT$ and any $y$ yields the second item and concludes the proof.
\end{proof}
\subsubsection{Type 3 solutions}
\label{sec:type 3}
These solutions are characterized by $J_1 > J_2 > J_3$ and
$$|\omega_1(t_0)| \neq \sqrt{\frac{J_3(J_2-J_3)}{J_1(J_1-J_2)}} |\omega_3(t_0)|$$
In this case the trajectory of $\omega(\cdot)$ is closed and thus periodic of a certain period $\tau > 0$, and not contained in a plane. Assume that condition~\eqref{eq:pe equivalent} is not satisfied. We apply the second item of Proposition~\ref{prop:preliminaire} with $T = \tau$. For any $\varepsilon$ small enough, there exists $t$ such that for all $s \in [t,t+\tau]$
\begin{displaymath}
J \omega(s) = R^T(s) \m
\end{displaymath}
remains between two parallel planes and distant by $\varepsilon\left|\m\right|$. As $\omega(\cdot)$ is $\tau-$ periodic, this is true for all ${s \in \mathbf R}$. When $\varepsilon$ goes to 0, we conclude that the trajectory of $\omega(\cdot)$ remains in a plane, which is a contradiction. Thus, condition~\eqref{eq:pe equivalent} is satisfied, unconditionally on $R(t_0)$.

\subsubsection{Type 4 solutions}
\label{sec:type 4}
\begin{figure}[!ht]
\includegraphics[width=\columnwidth]{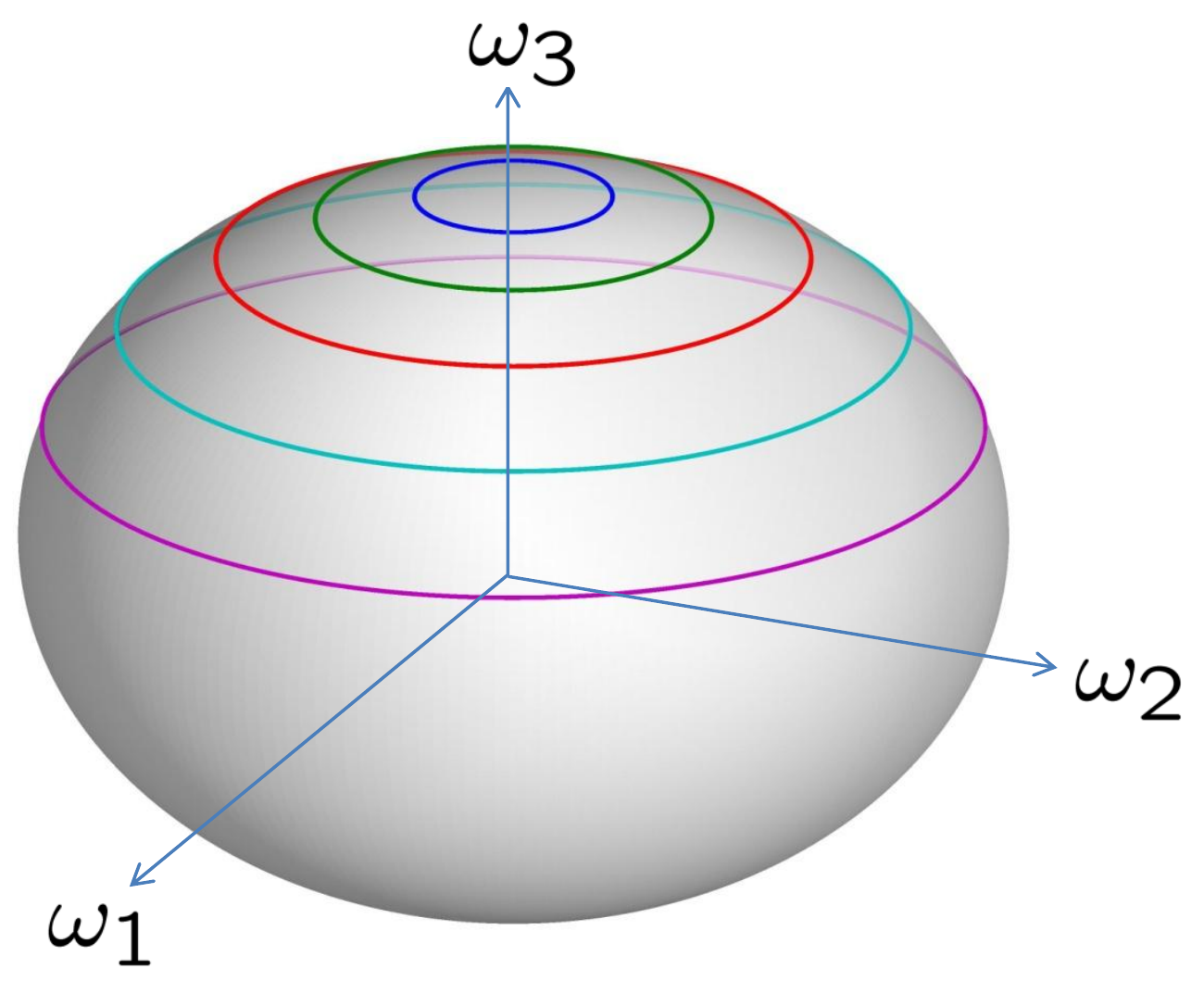}
\caption{Type 4 trajectories in the case $J_1 = J_2 > J_3$ on an ellipsoid corresponding to a constant $|\m|$}
\label{fig:type 4}
\end{figure}
We now consider the case where $\omega(t_0)$ is not an eigenvector of $J$ and two moments of inertia are equal. In this case the trajectory $t \mapsto \omega(t)$ is a circle, as represented in Figure~\ref{fig:type 4}. Since it is contained in a plane, we can not apply directly the same technique as in Section~\ref{sec:type 3}. Without loss of generality, we study the case $J_1 = J_2 > J_3$ (the case ${J_1 > J_2=J_3}$ is analog). We thus consider a trajectory $\omega$ such that $\omega(0)$ satisfies
\begin{displaymath}
\left(\omega_1(t_0),\omega_2(t_0)\right) \neq (0,0), \quad \omega_3(t_0) \neq 0
\end{displaymath}
Following the extensive analysis exposed in \cite{landau1982}, we conveniently chose the inertial frame $(e_1,e_2,e_3)$ so that $e_3$ is aligned with $\m$, namely
\begin{displaymath}
e_3 = \frac{\m}{\left|\m\right|}
\end{displaymath}
For this choice of $e_3$ and in the case where ${J_1=J_2}$, equations \eqref{eq R}-\eqref{eq euler} simplify considerably and one can show that the rotation matrix satisfies for all $t$
\begin{equation}
\label{eq:R deux moments egaux}
R(t) = p \left(\begin{array}{ccc}
(\dots) & (\dots) &  \cc \xi_1 (t-t_1) \\
(\dots) & (\dots) &  \ss \xi_1 (t-t_1)\\
 \cc \xi_2 (t-t_2) &  \ss \xi_2 (t-t_2) & \frac{\sqrt{1-p^2} }{p}
\end{array}
\right)
\end{equation}
where $(\dots)$ designates terms that are irrelevant in the following analysis, $t_1,t_2$ are constant and
\begin{align*}
p & \triangleq \sqrt{\frac{J_1 ^2 \omega_1(t_0)^2 + J_1^2 \omega_2(t_0)^2}{J_1 ^2 \omega_1(t_0)^2 + J_2^2 \omega_2(t_0)^2 +J_3^2 \omega_3(t_0)^2} } \quad \in (0,1) \\
\xi_1 & \triangleq \sqrt{\omega_1(t_0)^2 + \omega_2(t_0)^2 + \frac{J_3^2}{J_1^2} \omega_3(t_0)^2} \quad > 0 \\
\xi_2 & \triangleq \left(\frac{J_3}{J_1}-1\right) \omega_3(t_0) \quad \neq 0
\end{align*}
We now show that condition~\eqref{eq:pe equivalent} is satisfied by contradiction. Assuming that it is not, one can apply the first item of Proposition~\ref{prop:preliminaire} with $$T = \max \left(\frac{2\pi}{\xi_1}, \frac{2\pi}{|\xi_2|}\right)$$ For $\varepsilon$ small enough, there exists $t$ such that for all ${s \in [t,t+T]}$ $R(s)e_3$ remains between two planes orthogonal to $\ao$ and distant by $\varepsilon$. Moreover, expression~\eqref{eq:R deux moments egaux} yields for all $s$
\begin{displaymath}
R(s)e_3 = \left(\begin{array}{c}
p \cos \xi_1(s-t_1) \\
p \sin \xi_1(s-t_1) \\
\sqrt{1-p^2}
\end{array}\right)
\end{displaymath}
Simple geometric considerations show that
\begin{displaymath}
\sqrt{1-(\ao^Te_3)^2} \leq \frac{\varepsilon}{2p}
\end{displaymath}
which yields $\ao = \pm e_3$ when $\varepsilon$ goes to $0$. Hence for $\varepsilon$ small enough, and all $s \in[t,t+T]$
\begin{displaymath}
R(s)e_1 = \left(\begin{array}{c}
(\dots) \\
(\dots) \\
p \cos \xi_2 (s-t_2)
\end{array}
\right)
\end{displaymath}
remains between two planes orthogonal to $\ao = \pm e_3$. Taking $\varepsilon < 2p$ yields a contradiction. The trajectories $R(t)e_1$ and $R(t)e_3$ are represented in Figure~\ref{fig:axes} for better visual understanding of the proof.
\begin{figure}[!ht]
\includegraphics[width = \columnwidth]{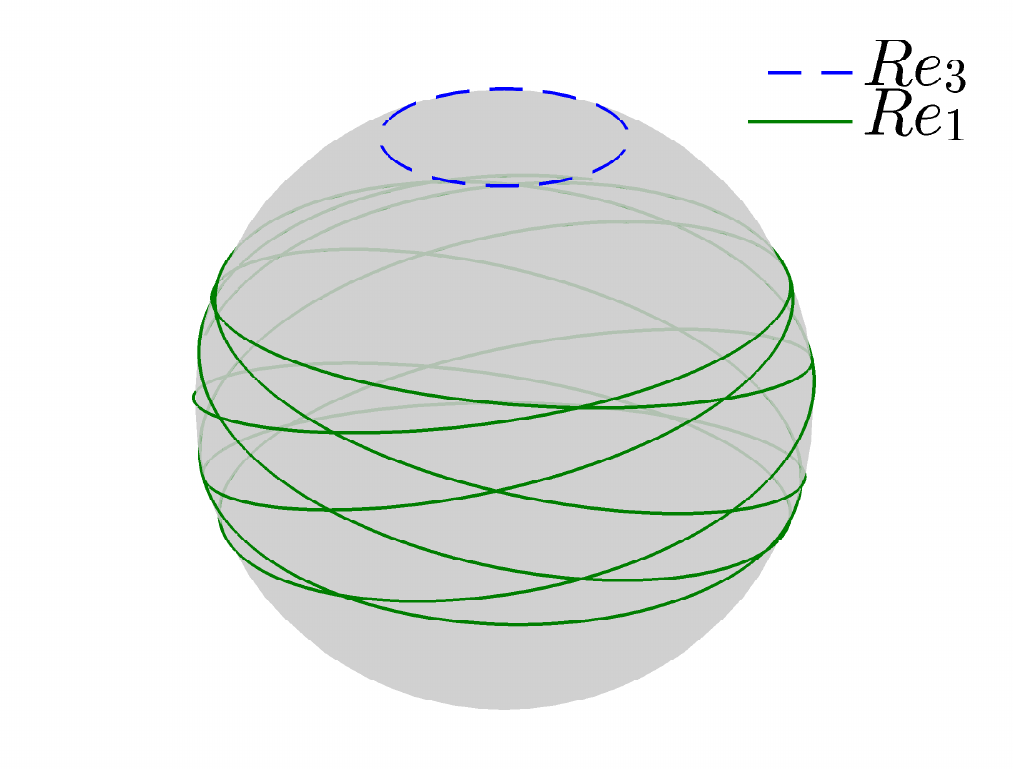}
\caption{$R(t)e_3$ (dashed) and $R(t)e_1$ (solid) evolving on the unit sphere}
\label{fig:axes}
\end{figure}
\subsection{Conclusion}
In this section we have shown the following result.
\begin{theorem}
\label{thm:PE presque toujours vrai}
Consider the vector $$a(t) = R(t)^T \ao $$ where $R(t)$ is a rotation matrix defined as the solution of the free-rotation dynamics \eqref{eq R}-\eqref{eq euler} with $\tau= 0$. Assumption~\ref{hyp:PE} is satisfied for almost all initial conditions ${(R(t_0),\omega(t_0))}$. It fails only in the cases listed below
\begin{enumerate}[(i)]
\item $\omega(t_0)$ is an eigenvector or $J$ and $R(t_0)J \omega(t_0)$ is aligned with $\ao$, or
\item the eigenvalues of $J$ are of the form ${J_1 > J_2 > J_3}$, the coordinates of $\omega(t_0)$ in the trihedron of orthonormal eigendirections of $J$ satisfy
\begin{equation}
\label{eq : condition gamma}
\left|\omega_1(t_0)\right| = \sqrt{\frac{J_3(J_2-J_3)}{J_1(J_1-J_2)}} \left|\omega_3(t_0)\right|
\end{equation} and $R(t_0)J \omega(t_0)$ is aligned with $\ao$.
\end{enumerate}
\end{theorem}
It follows that, except for the initial conditions listed in items $(i),(ii)$, the conclusion of Theorem~\ref{thm convergence obs} holds without requiring Assumption 2, which is automatically satisfied. Therefore, in almost all cases, observer~\eqref{def : observer} asymptotically reconstructs the desired angular velocity $\omega$.

\section{Simulation results}
\label{sec : simu}
In this section we illustrate the convergence of the observer and sketch the dependence with respect to the tuning gain $k$.

Simulations were run for a model of a CubeSat \cite{cubesat2014}. The rotating rigid body under consideration is a rectangular
parallelepiped of dimensions about ${20~\textrm{cm} \times 10~\textrm{cm} \times 10~\textrm{cm}}$ and mass $2$kg assumed to be slightly non-homogeneously distributed. The resulting moments of inertia are
\begin{displaymath}
J_1 = 87~\textrm{kg.cm}^2, \quad J_2 = 83~\textrm{kg.cm}^2, \quad J_3 = 37~\textrm{kg.cm}^2
\end{displaymath}
No torque is applied on this system, which is thus in free-rotation. Referring to Section~\ref{sec:PE}, we will consider Type 1 and Type 3 trajectories.

In this simulation the reference unit vector is the normalized magnetic field $\ao$. The satellite is equipped with 3 magnetometers able to measure the normalized magnetic field $y_a$ in a magnetometer frame $\Rm$.

It shall be noted that, in practical applications, the sensor frame $\Rm$ can differ from the body frame $\Rb$ (defined along the principal axes of inertia) through a constant rotation $\Rmb$. With these notations, we have
\begin{displaymath}
a = \Rmb^T y_a
\end{displaymath}
which is a simple change of coordinates of the measurements.

For sake of accuracy in the implementation, reference dynamics and state observer~\eqref{def : observer}
were simulated using Runge-Kutta 4 method with sample period $0.01$s. The generated trajectories correspond to ${\om \simeq 100}$~[rad/s].

\subsection{Noise-free simulations}
To emphasize the role of the tuning gain $k$, we first assume that the sensors are perfect i.e. without noise. Typical measurements for a general Type 3 trajectory are represented in Figure~\ref{fig:mesures}. As $J_1$ and $J_2$ are almost equal, the third coordinate is almost (but not exactly) periodic.
\begin{figure}[!t]
\includegraphics[width = \columnwidth]{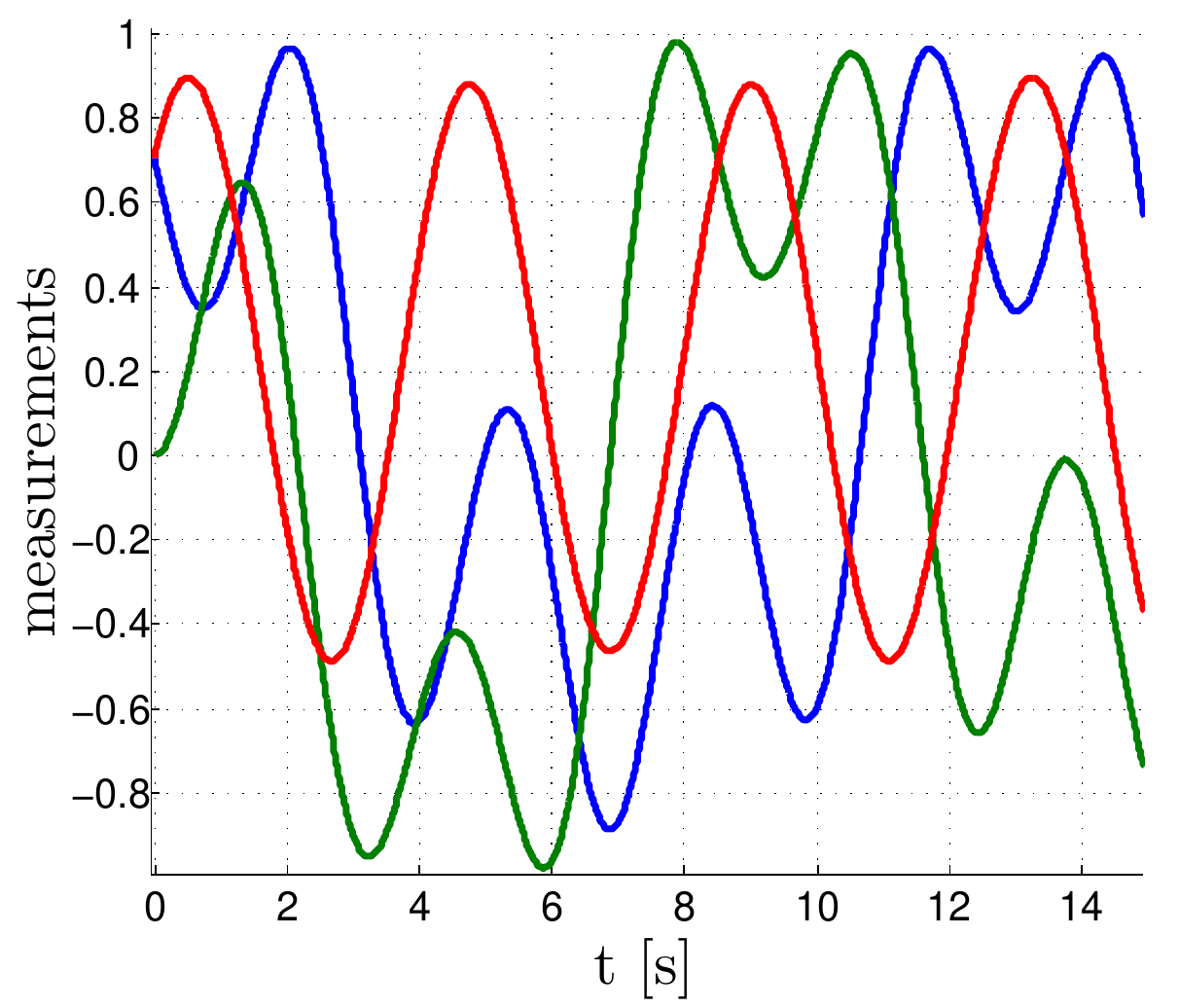}
\caption{Typical measurements in the ideal noise-free case}
\label{fig:mesures}
\end{figure}
Figure~\ref{fig:convergence2} shows the convergence of the observer for various values of $k$.
\begin{figure}[!h]
\includegraphics[width=\columnwidth]{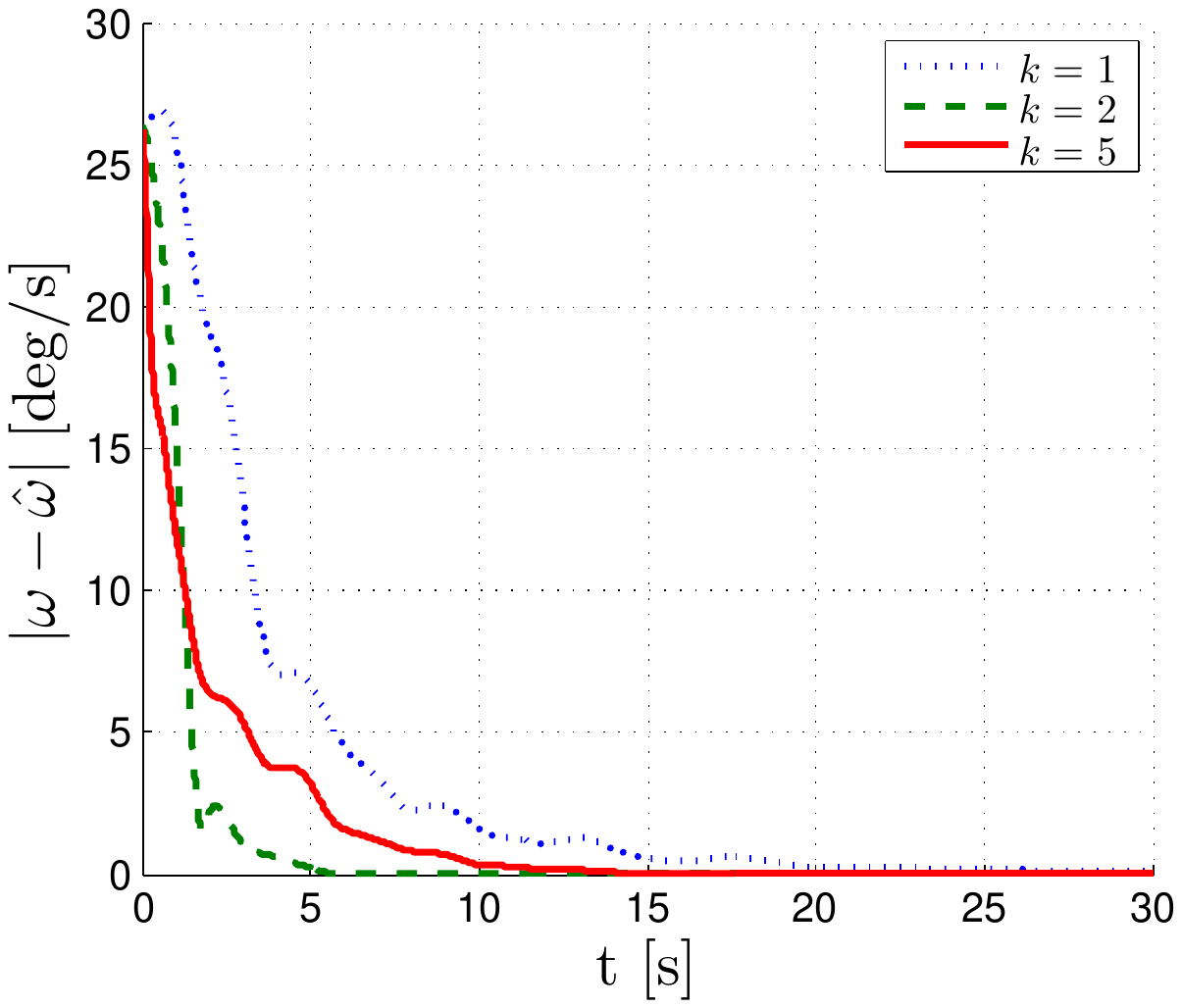}
\caption{Convergence of the observer for increasing values of $k$}
\label{fig:convergence2}
\end{figure}
Interestingly, large values of $k$ produce undesirable effects. This is a structural difference with the two reference vectors based observer previously introduced by the authors~\cite{magnis2014a}. The reason is that the convergence is guaranteed by a PE argument and not by a uniformly negative bound on eigenvalues.

In Figure~\ref{fig:nonPE} we represented the observer error for a case where the PE assumption is not satisfied, namely for a constant $\omega$ with $\m$ and $\ao = (1,0,0)$ aligned. This is a singular case, as discussed earlier. Interestingly, the coordinates $\widetilde \omega_2$ and $\widetilde \omega_3$ converge to zero, while $\widetilde \omega_1$ converges to a constant value. This can easily be proved by using LaSalle invariance principle. Indeed, in this case, $\omega$ is constant and the measurements $a(\cdot)$ satisfy a LTI differential equation.
\begin{figure}[!ht]
\includegraphics[width=\columnwidth]{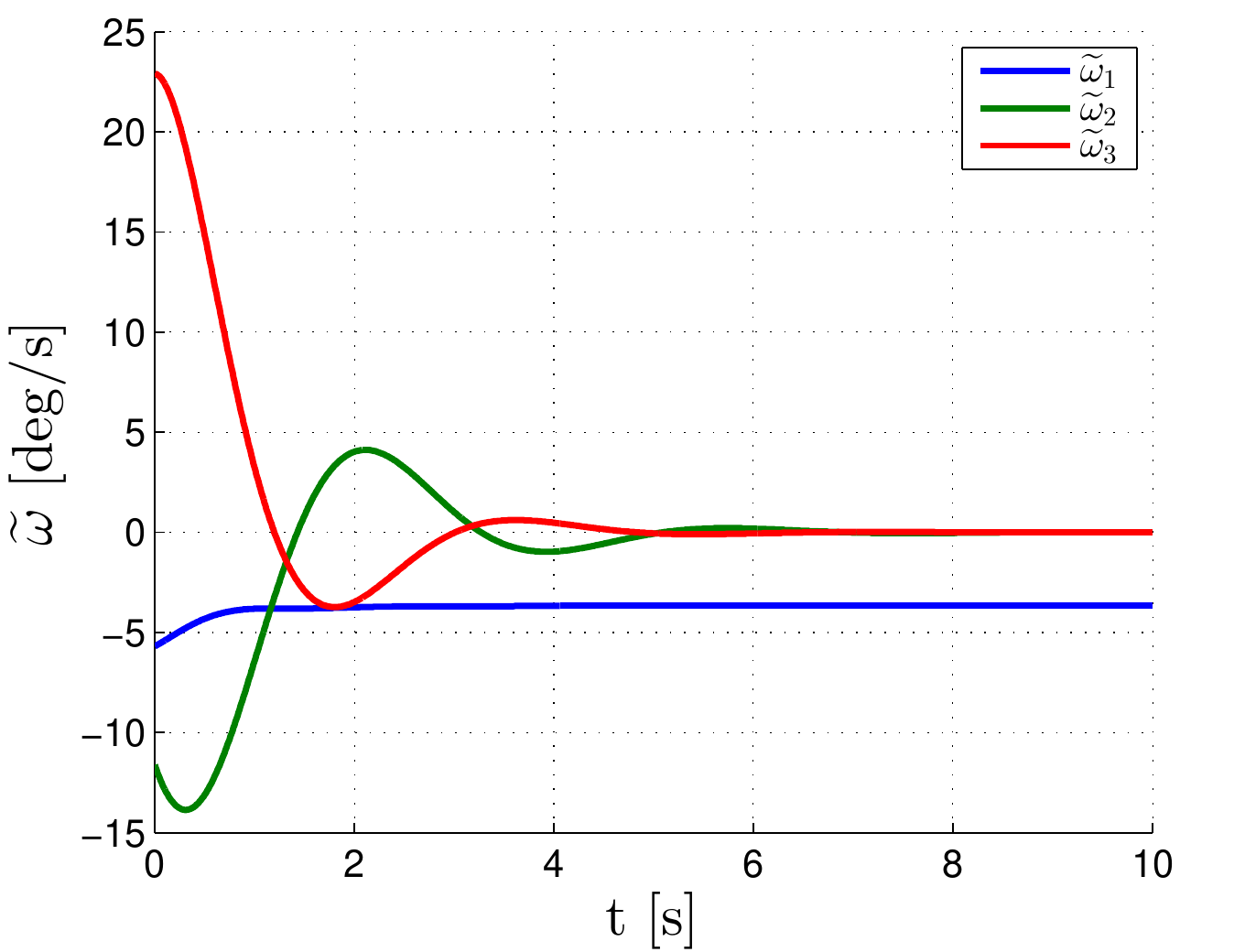}
\caption{Without the PE assumption asymptotic convergence of the observer is lost, a bias remains.}
\label{fig:nonPE}
\end{figure}

\subsection{Measurement noise}
We now study the impact of measurement noise on the observer performance. The simulation parameters remain the same but we add Gaussian measurement noise with standard deviation $\sigma = 0.03~[\textrm{Hz}^{-\frac{1}{2}}]$. Typical measurements are represented in Figure~\ref{fig:mesures bruitees}.
\begin{figure}[!ht]
\includegraphics[width = \columnwidth]{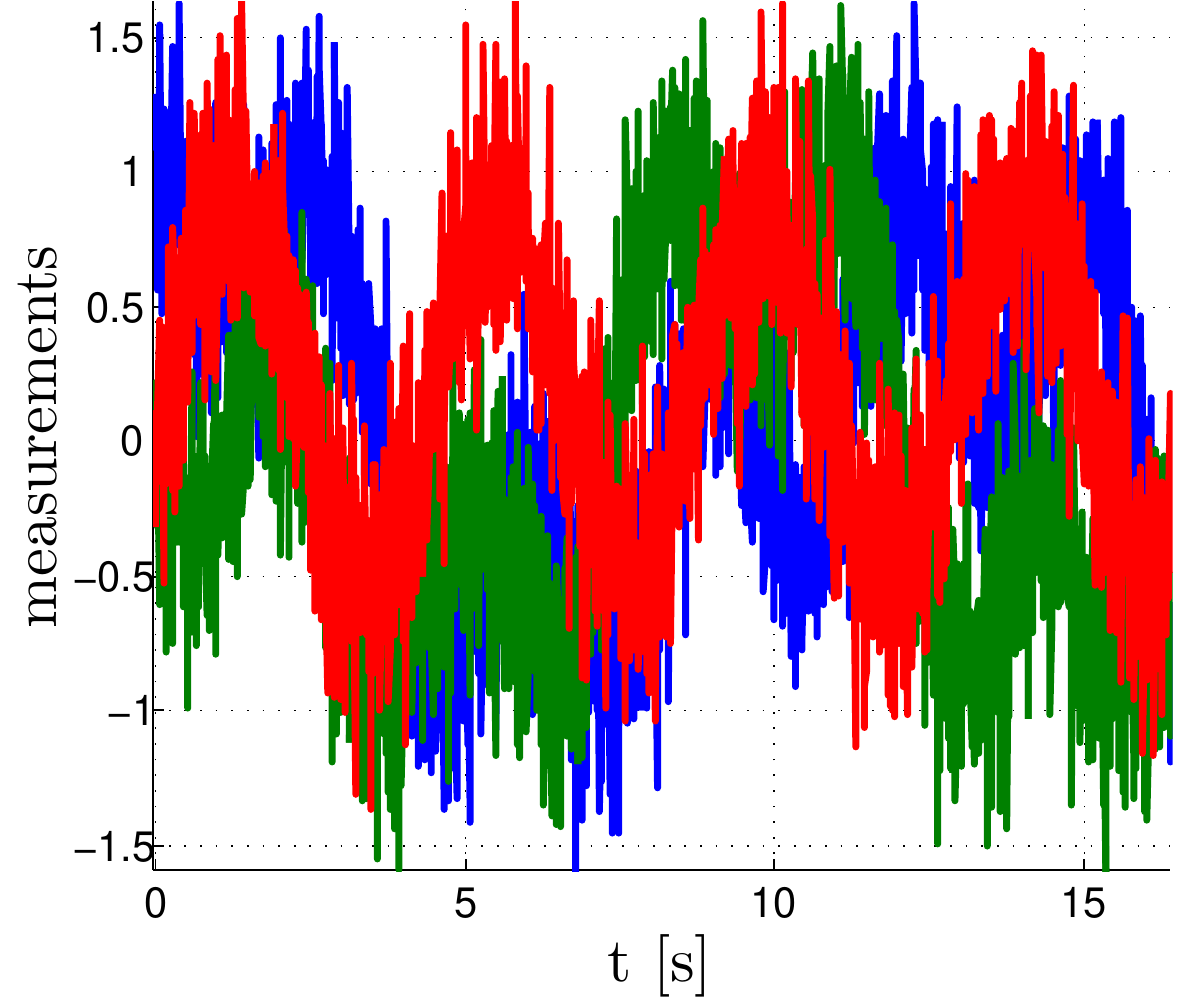}
\caption{Vector measurement with additive noise}
\label{fig:mesures bruitees}
\end{figure}
The observer yields a residual error, about $5 \%$ in Figure~\ref{fig:convergence bruit} for $k=1$. Note that the measurement noise is filtered, thanks to a relatively low value of the gain $k$. For large values of $k$, the observer does not converge anymore (not represented).
\begin{figure}[!ht]
\includegraphics[width = \columnwidth]{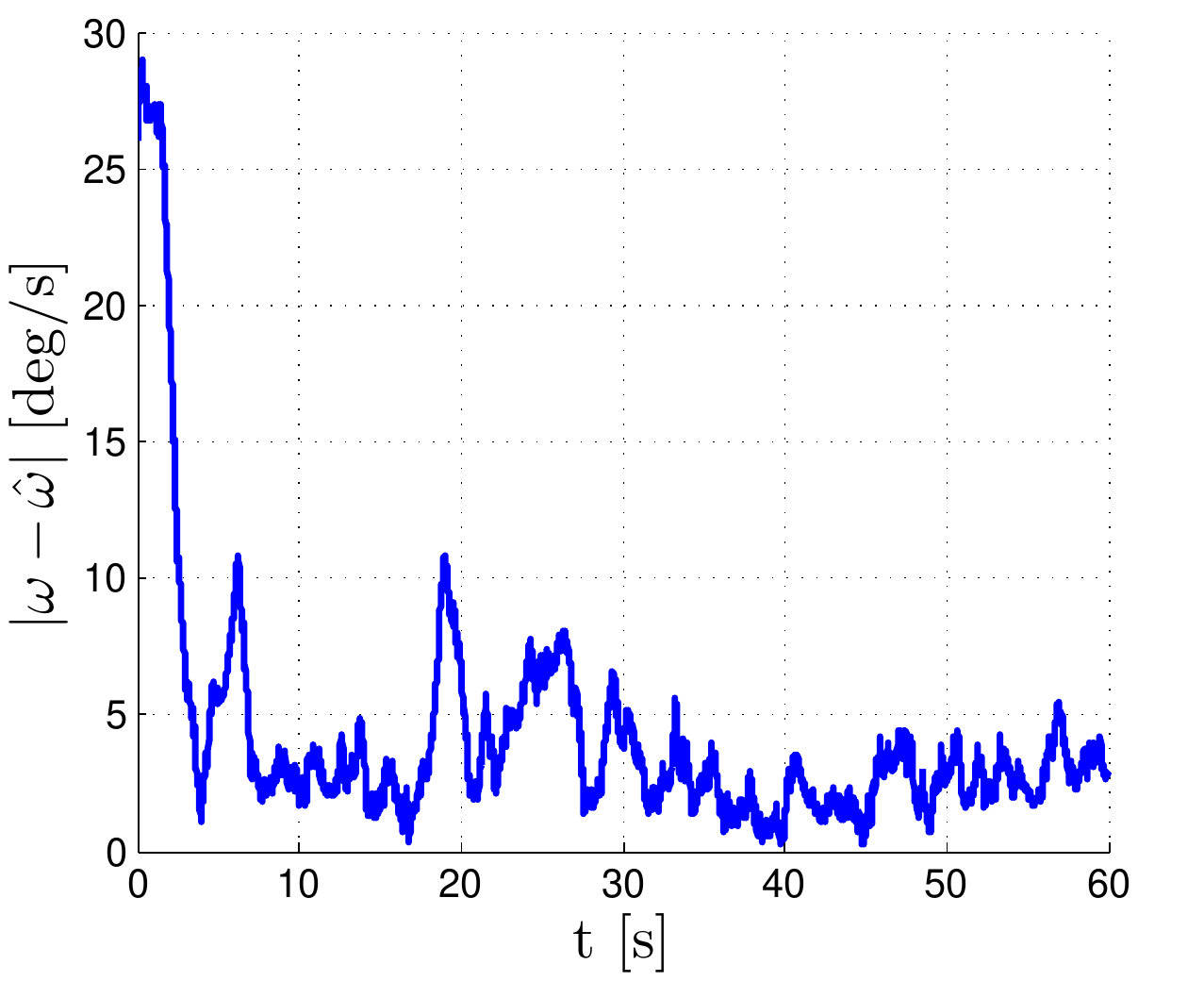}
\caption{Observer performance under noisy measurement for $k=1$}
\label{fig:convergence bruit}
\end{figure}

\section{Conclusions and perspectives}
\label{sec : conclusion}
A new method to estimate the angular velocity of a rigid body has been proposed in this article. The method uses onboard measurements of a single constant vector. The estimation algorithm is a nonlinear observer which is very simple to implement and induces a very limited computational burden. At this stage, an interesting (but still preliminary) conclusion is that, in the cases considered here, rate gyros could be replaced with an estimation software employing cheap, rugged and resilient sensors. In fact, any type of sensors producing a 3-dimensional vector of measurements such as e.g., Sun sensors, magnetometers, could constitute one such alternative. Assessing the feasibility of this approach requires further investigations including experiments.

More generally, this observer should be considered as a first element of a class of estimation methods which can be developed to address several cases of practical interest. In particular, the introduction of noise in the measurement and uncertainty on the input torque (assumed here to be known) will require extensions such as optimal filtering to treat more general cases. White or colored noises will be good candidates to model these elements. Also, slow variations of the reference vector $\ao$ should deserve particular care, because such drifts naturally appear in some cases.

On the other hand, one can also consider that this method can be useful for other estimation tasks. Among the possibilities are the estimation of the inertia $J$ matrix which we believe is possible from the measurements considered here. This could be of interest for the recently considered task of space debris removal~\cite{bonnal2013}.
Finally, recent attitude estimation techniques have favored the use of vector measurements \emph{together} with rate gyros measurements as inputs.
Among these approaches, one can find {\emph{i)} Extended Kalman Filters (EKF)-like algorithms e.g. \cite{choukroun2006,schmidt2008}, {\emph{ii)} nonlinear observers~\cite{mahony2008,martin2010,vasconcelos2008,tayebi2011,grip2011,trumpf2012}.
This contribution suggests that, here also, the rate gyros could be replaced with more in-depth analysis of the vector measurements.


\end{document}